\documentclass[a4paper,12pt]{article}
\usepackage[margin=3cm,footskip=1cm]{geometry}

\usepackage{lmodern}
\usepackage{dsfont}
 \usepackage[utf8]{inputenc} 
\usepackage[T1]{fontenc}
\usepackage{microtype}
\usepackage{syntonly}
\usepackage{float} 
\usepackage{xcolor}
\usepackage{amsmath,amssymb,amsfonts,amsthm,bm,mathtools,xspace,booktabs}
\usepackage{graphicx,subfig,color}
\usepackage{pdfcomment}
\usepackage{enumitem}
\usepackage[normalem]{ulem}
\usepackage[mathlines,pagewise]{lineno}

\usepackage{authblk}

\let\oldabstract\abstract
\makeatletter
\renewcommand\abstract{%
  \providecommand\keywords{\par\medskip\noindent\textit{Keywords:}\xspace}
  \oldabstract\noindent\ignorespaces}
\makeatother

\numberwithin{equation}{section}

\theoremstyle{plain} 
\newtheorem{theorem}{Theorem}[section]
\newtheorem{lemma}[theorem]{Lemma}
\newtheorem{proposition}[theorem]{Proposition}
\newtheorem{definition}[theorem]{Definition}
\newtheorem{corollary}[theorem]{Corollary}

\usepackage{algorithmic, algorithm}

\def\R{\mathbb R}
\def\E{\mathbb E}

\def\F{\mathcal F}
\def\K{\mathcal K}
\def\N{\mathbb N}

\def\1{\mathbf{1}}

\def\X{\mathbf{X}}

\def\gn{\widehat{g}_n}

\def\wrho{{\widehat{\rho}_n}}

\def\Dn{{|D_n|}}
\def\convl{\xrightarrow[n\rightarrow +\infty]{distr.}}
\def\convP{\xrightarrow[n\rightarrow + \infty]{\mathbb{P}}}

\def\cum{\mathrm{Cum}}
\graphicspath{{.}{./figure_quantifying_rep_DPP_supplementary/}{./figure_quantifying_rep_DPP/}{./figure_intro/}}

\begin{document}

\title{Brillinger mixing of determinantal point processes and statistical applications}

\author[1]{Christophe Ange Napol\'eon Biscio}
\author[1,2]{Fr\'ed\'eric Lavancier}
\affil[1]{ Laboratoire de Math\'ematiques Jean Leray\\
University of Nantes, France}
 \affil[2]{ Inria, Centre Rennes Bretagne Atlantique, France}

\date{}

\maketitle

\begin{abstract}
Stationary determinantal point processes are proved to be Brillinger mixing. This property is an important  step towards asymptotic statistics for these processes. As an important example, a central limit theorem for a wide class of functionals of determinantal point processes is established. This result yields in particular the asymptotic normality of the estimator of the intensity of a stationary determinantal point process  and of the kernel estimator of its pair correlation. 
%
   \end{abstract}

\section{Introduction}

Determinantal point processes (DPPs) are models for repulsive point patterns, where 
nearby points of the process tend to repel each other. They have been introduced in their general form in~\cite{macchi1975coincidence} and extensively studied in Probability theory, see~\cite{hough2009zeros} and \cite{Soshnikov:00}.  From a statistical perspective, DPPs have been applied in machine learning \cite{Kulesza:Taskar:12}, spatial statistics \cite{lavancierpublish,lavancier_extended} and telecommunication  \cite{deng2014ginibre,Miyoshi:Shirai13}.
The growing interest for DPPs  in the statistical community is due to their appealing properties: They can be quickly and perfectly simulated, parametric models can easily be constructed,  their moments are known and the likelihood has a closed form expression. Their definition and some of their properties are recalled in Section~\ref{rappelDPP} and we refer to~\cite{lavancierpublish} for more details. Some realizations are showed in Figure~\ref{fig:samples}.

We focus in this paper on stationary DPPs on the continuous  space $\R^d$ and we prove that they are Brillinger mixing. To the best of our knowledge, no mixing property was established so far for DPPs. The Brillinger mixing property is an important step towards asymptotic statistics for DPPs, which are mainly unexplored in the literature.
The definition, recalled in Section~\ref{section preliminaire brillinger}, is based on the moments of the process. 
Specifically, a stationary point process is Brillinger mixing if for any $k\geq 2$ the total variation of its reduced factorial cumulant measure of order $k$ is finite, see for instance~\cite{heinrichasymptotic:13} or~\cite{krickeberg1980}. Already known Brillinger mixing point processes include Poisson cluster processes and Mat\'ern hardcore point processes (of type I, type II and some generalizations as in~\cite{stoyangeneralizedmatern}), see \cite{heinrichschmidt1985shot} and \cite{heinrichpawlasabsolute}. As far as we know, the Mat\'ern hardcore models are the only models of repulsive stationary point processes that have been proved to be Brillinger mixing. Our result shows that DPPs provide a new flexible class of repulsive Brillinger mixing point processes.

In Section~\ref{section statistical application brillinger}, we give some applications of the Brillinger mixing property of DPPs. These are mainly based on general results established in~\cite{jolivetTCL}, \cite{heinrichstellaintegrated} and \cite{StellaKleinempirical2011}, that we extend and/or simplify in the setting of stationary DPPs.  Namely, we prove the asymptotic normality of a wide class of functionals of order $p$ of a DPP, in the spirit of~\cite{jolivetTCL}. This result allows in particular to retrieve the asymptotic behavior of the estimator of the intensity of a DPP, known since~\cite{SoshnikovGaussianLimit}, and to get the asymptotic normality of the kernel estimator of the pair correlation function of a DPP, which is a new result presented in Section~\ref{pcf}.  The Brillinger mixing property is useful for many other applications, see for instance~\cite{HeinrichProkesova:10}, \cite{jolivetuppermoment} and \cite{jolivet88momentestimation}. 
In an ongoing project \cite{biscio_contrast}, this property is used to get the asymptotic normality of minimum contrast estimators for parametric DPPs.

The reminder of this paper is organized as follows. Section~\ref{preliminaries} gathers some basic facts about stationary DPPs,  moment measures of a point process and the Brillinger mixing property.  Our main result stating that stationary DPPs are Brillinger mixing is presented in Section~\ref{section brillinger mixing of DPP}. Some statistical applications are given in~Section~\ref{section statistical application brillinger}. Section~\ref{section proof TCL brillinger} and Section~\ref{section preuve clt gn} contain some technical proofs and Section~\ref{appendix brillinger} is an appendix dealing with the computation of the asymptotic variance in the  statistical applications of Section~\ref{section statistical application brillinger}.

\section{Preliminaries}\label{preliminaries}

\subsection{Determinantal point processes}\label{rappelDPP}
For $d\geq 1$, we denote by 
$\mathcal{B}_0(\R^d)$ the class of bounded Borel sets on $\R^d$. 
For $\mathbf{x} \subset \R^d$ and $B\in \mathcal{B}_0(\R^d)$, $\mathbf{x}(B)$  stands for  the number of points in $\mathbf{x}\cap B$. We let $\mathcal{N} := \lbrace \mathbf{x} \subset \R^d,\ \mathbf{x}(B)< \infty,\ \forall B\in \mathcal{B}_0(\R^d) \rbrace$ be  the space of locally finite configurations of points in $\R^d$. This set is equipped with the $\sigma$-algebra generated by the sets $ \lbrace  \mathbf{x} \subset \R^d,\ \mathbf{x}(B)=n \rbrace$ for all $ B\in \mathcal{B}_0(\R^d) $ and all $n\in \N\cup\{0\}$, where $\N$ denotes the space of positive integers.
A point process on $\R^d$ is a measurable application from a probability space 
 into the set $\mathcal{N}$. We denote a point process by a bold capital letter, usually $\X$, and identify the mapping $\X$ and the associated random set of points. 
All considered point processes are assumed to be simple, i.e. two points of the process never coincide, almost surely. For further details on point processes, we refer to~\cite{daleyvol1,daleyvol2}. 


The factorial moment measures and especially the joint intensities of order $k$ of a point process, defined below, are important quantities of interest. They in particular characterize the law of determinantal point processes.  
\begin{definition}\label{definition factorial moment intro}
The factorial moment measure of order $k$ ($k\geq 1$) of a simple point process $\X$  is the measure on $\R^{dk}$, denoted by $\alpha^{(k)}$, such that for any family of subsets $D_1,\dots,D_k$ in $\R^d$,
\begin{align*}
\alpha^{(k)} \left( D_1\times \ldots \times D_k \right) = \E \left( \sum^{\neq}_{ (x_1,\ldots,x_k) \in \X^k} \1_{\lbrace x_1\in D_1, \ldots,  x_k\in D_k \rbrace }  \right)
\end{align*}
where $\E$ is the expectation over the distribution of $\X$ and the symbol $\neq$ over the sum means that we consider only mutually disjoints $k$-tuples of points $x_1,\ldots,x_k$.

If $\alpha^{(k)}$ admits a density with respect to the Lebesgue measure on $\R^{kd}$, this density is called the joint intensity of order $k$  of $\X$ and is denoted by $\rho^{(k)}$.
\end{definition}

Important particular cases are the factorial moment measure of order one, called the intensity measure, and the factorial moment measure of order two. 
If $\X$ is stationary, then 
for  all $S\subset \R^d$, there exists $\rho>0$ such that $\alpha^{(1)}(S) = \rho |S|$, where $|S|$ stands for the volume (Lebesgue measure) of $S$. In this case, for any $x\in\R^d$, $\rho^{(1)}(x)=\rho$ is called the \textit{intensity} of the process and represents the expected number of points per unit volume. 
Regarding the joint intensity of order two, for $(x,y) \in \R^{2d}$ and $x\neq y$, $\rho^{(2)}(x,y)$ may be viewed heuristically as  the probability that there is a point of the process in a small neighbourhood around $x$ and another point in a small neighbourhood around $y$. 
 In spatial statistics, the second order properties of a point process are often studied through the \textit{pair correlation function }(pcf).  The pcf is defined  for almost every $(x,y)\in\R^{2d}$ by 
\begin{align*}
 g(x,y) = \frac{\rho^{(2)}(x,y)}{\rho(x)\rho(y)}.
\end{align*}
In the stationary and isotropic case, $g(x,y)=g_0(r)$ depends only  on the Euclidean distance $r=|x-y|$. Intuitively, $g_0(r)$ is the quotient of the probability that two points occur at distance $r$ (taking into account the interaction induced by the process) and the same probability if there was no interaction. Consequently, for $r>0$, a common interpretation, see for instance~\cite{stoyan1987stochastic}, is that $g_0(r)>1$  characterizes clustering at distance $r$ while $g_0(r)<1$ characterizes repulsiveness at distance $r$. 

%
%

\medskip

Determinantal point processes (DPPs) are defined through their joint intensities. They have been introduced in their current form by Macchi in~\cite{macchi1975coincidence} to model the position of particles that repel each other. 
Since our results concern only stationary DPPs, we restrict the definition to this subclass, which simplifies the notation. 

\begin{definition}\label{DPPdefinition intro}
Let $C: \R^d \rightarrow \R$ be a function.
A point process $\X$ on $\R^d$ is  a stationary DPP with kernel $C$, in short $\X\sim DPP(C)$,  if for all $k\geq 1$ its joint intensity of order $k$ satisfies the relation
\begin{align*}
 \rho^{(k)}(x_1,\ldots x_k)=\det [C](x_1,\dots,x_k)
\end{align*}
for almost every $(x_1,\dots,x_k)\in \R^{dk}$, where $[C](x_1,\dots,x_k)$ denotes the matrix with entries $C(x_i-x_j)$,  $1\leq i,j\leq k$.
\end{definition}

It is actually possible to consider complex-valued kernels and/or non-stationary DPPs, but this is not the setting of this paper and we refer to~\cite{hough2009zeros} for a review on DPPs in the general case. The existence of a DPP requires several conditions on the kernel $C$. Sufficient conditions in the stationary case are provided in the next proposition. They rely on the Fourier transform of $C$ and are easy to verify in practice, unlike the general conditions for non stationary DPPs, see~\cite{hough2009zeros}. 

We define the Fourier transform of a function $h\in L^1(\R^d)$ as 
\begin{align*}
  \F(h)(t)=\int_{\R^d} h(x) e^{2i\pi x\cdot t}dx,\quad \forall t \in \R^d
\end{align*}
 and extend this definition to  $L^2(\R^d)$ by Plancherel's theorem, see~\cite{stein1971fourier}. We have the following existence result.

\begin{proposition}[\cite{lavancierpublish}]\label{DPPexistence intro}
Assume $C$ is a symmetric continuous real-valued function in $L^2(\R^d)$. Then $DPP(C)$ exists if and only if  $0\leq \F(C)\leq 1$.
\end{proposition}
In other words, by Proposition~\ref{DPPexistence intro} any continuous real-valued covariance function $C$ in $L^2(\R^d)$ with $\F(C)\leq 1$ defines a DPP.  Henceforth, we assume the following condition.\medskip

\noindent{\bf Condition $\K(\rho)$.} A kernel $C$ is said to verify condition $\K(\rho)$ if $C$ is a symmetric continuous real-valued function in $ L^2(\R^d)$ with $C(0)=\rho$ and $0\leq \F(C)\leq 1$.\medskip


By definition, all moments of a DPP are explicitly known. In particular, assuming  $\K(\rho)$,  $DPP(C)$ is stationary with intensity $\rho$ and denoting  $g$ its pcf  we have 
 \begin{align}\label{DPPpcf}
g(x,y)=1 - \frac{C(x-y)^2}{\rho^2}
\end{align}
for almost every $(x,y)\in\R^{2d}$. Consequently $g\leq 1$, which shows that DPPs exhibit repulsiveness.

\medskip

A first example of stationary DPP is  the stationary Poisson process with intensity $\rho$, which corresponds to  the kernel
$C(x)=\rho \1_{\lbrace x=0 \rbrace}$. 
However, this example is very particular and represents in some sense the extreme case of a  DPP without any interaction. In particular its kernel does not satisfy $\K(\rho)$ since it is not continuous. In contrast,  $\K(\rho)$ is verified by numerous covariance functions , and this makes easy the definition of parametric families of DPPs, where the condition $\F(C)\leq 1$ implies some restrictions on the parameter space. Some  examples are  given in~\cite{lavancierpublish} and \cite{bisciobernoulli},  where the stationary Poisson process appears as a degenerated case. For instance, the Gaussian kernels correspond to   $C(x) = \rho e^{-\left| x/\alpha \right|^2}$,  $x\in\R^d$, where the existence condition implies $\alpha \leq 1/(\sqrt{\pi}\rho^{1/d})$. 
Another important example is the most repulsive stationary DPP with intensity $\rho$, as defined and determined in~\cite{bisciobernoulli}. Its kernel $C$ is the Fourier transform of the indicator function of the Euclidean ball centered at the origin with volume $\rho$, which gives
\begin{align}\label{jinc}
 C(x)= \frac{\sqrt{\rho\Gamma(\frac{d}{2}+1)}}{\pi^{d/4}} \frac{J_{\frac{d}{2}}\left(2\sqrt{\pi}\Gamma(\frac{d}{2}+1)^{\frac{1}{d}}\rho^{\frac{1}{d}} |x|\right)}{|x|^{\frac{d}{2}}}, \quad \forall x\in\R^d,
\end{align}
where $J_{\frac{d}{2}}$ denotes the Bessel function of the first kind of order $\frac{d}{2}$. Some examples of realisations of DPPs are given in Figure~\ref{fig:samples}.

\begin{figure}[h]
\begin{center}
\begin{tabular}{ccc} 
 \includegraphics[scale=0.34]{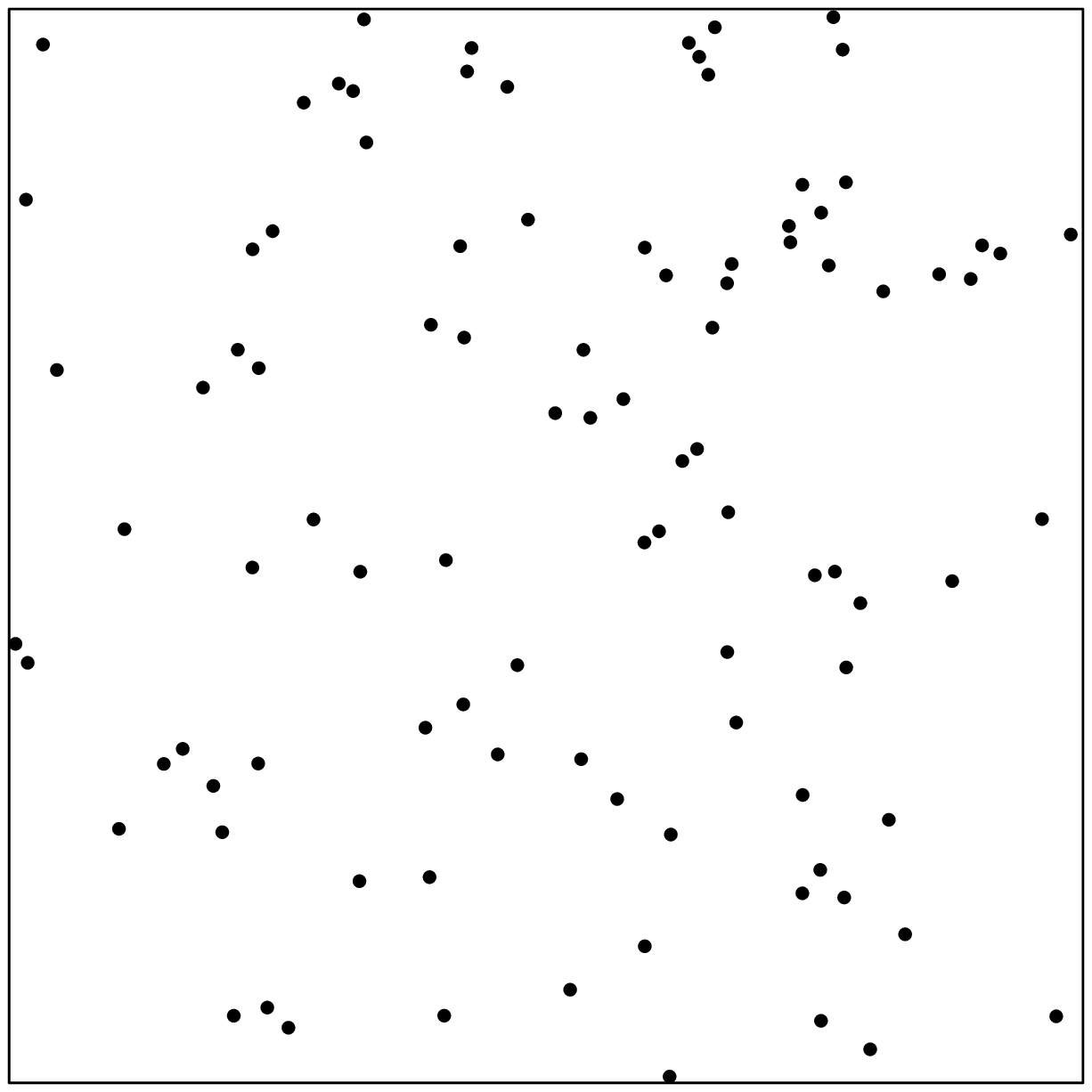}   &  \includegraphics[scale=0.34]{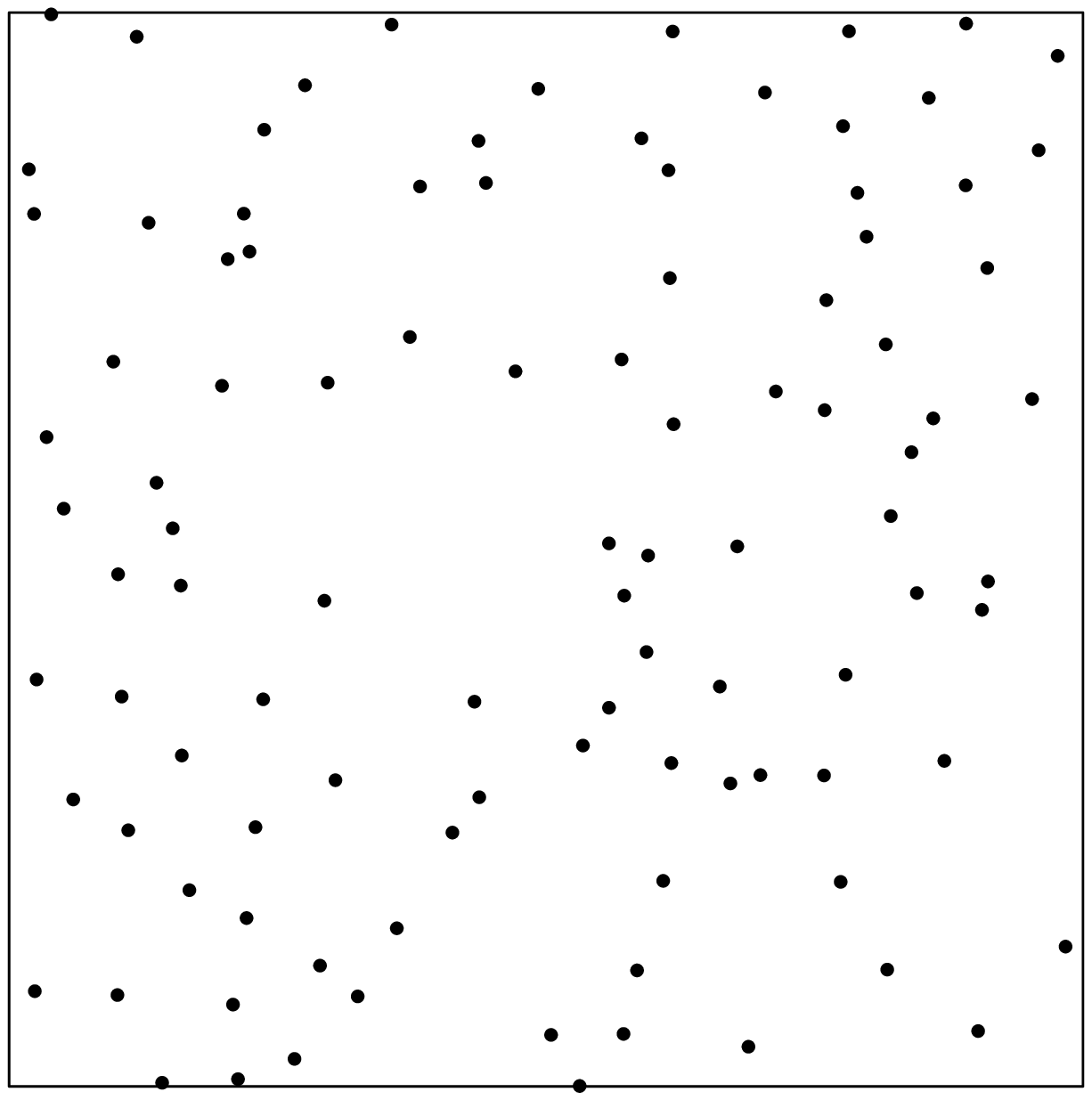} & \includegraphics[scale=0.34]{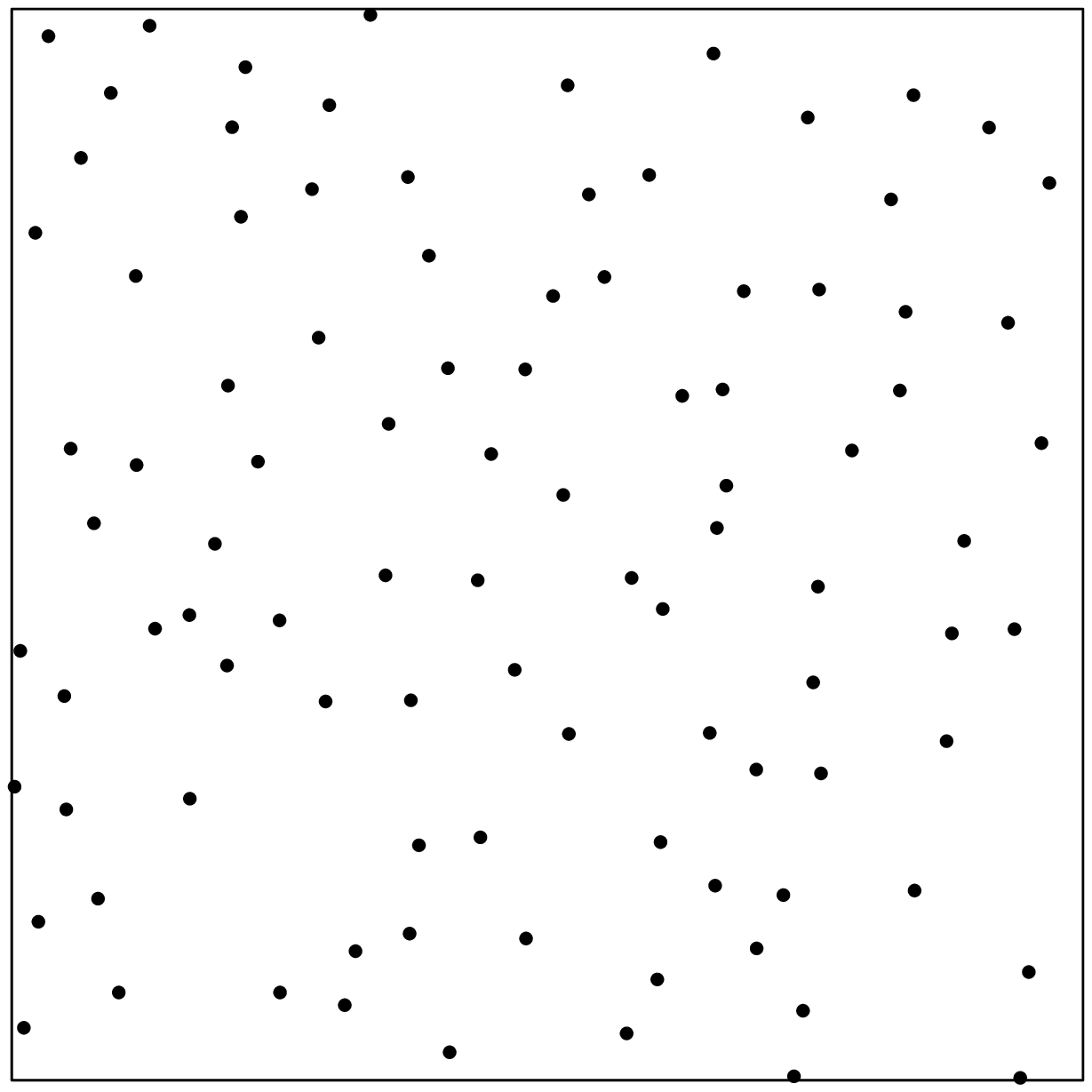}
\end{tabular}
\caption{From left to right,  letting the intensity $\rho=100$, realizations on $[0,1]^2$ of a stationary Poisson process, a DPP with a Gaussian kernel and the maximal  possible choice for the range parameter $\alpha$ ($\alpha=0.056$), a DPP with kernel \eqref{jinc}.}\label{fig:samples}
\end{center}
\end{figure}

\subsection{Moment measures and Brillinger mixing}\label{section preliminaire brillinger}

 In this section, we review the definition of the cumulant and factorial cumulant moment measures of a point process $\X$ as well as their reduced version. These are at the basis of the Brillinger mixing property defined in the following. The relation with the Laplace and the probability generating functionals of $\X$ is also described. We assume that for any bounded set $A$, the random variable $\X(A)$ has moments of any order. This ensures that the quantities introduced in this section are well defined. Note that by definition this assumption holds true for a DPP.  Further details on these topics may be found in \cite{daleyvol1,daleyvol2} and \cite{krickeberg1980}. 

\begin{definition}\label{definition cumulant}
 For $k\in \N$, the  cumulant  of the $k$ random variables $X_1,\ldots,X_k$ is, if it exists,
 \begin{align*}
  \cum(X_1,\ldots,X_k) = \left. \frac{\partial^k}{\partial t_1 \ldots \partial t_k} \log  \E \left[ \exp{\left(\sum_{i=1}^k t_i X_i\right)} \right] \right|_{t_1=\ldots=t_k=0}.
 \end{align*}
 The $k$-th order cumulant of the random variable $X$ is $\cum_k(X):=\cum(X,\ldots,X)$.
\end{definition}

The notion of cumulant of random variables extends to point processes as follows.

\begin{definition}\label{definition cumulant measure} 
 For $k\in \N$, the $k$-th order  cumulant moment measure $\gamma_{k}$ of a point process $\X$ is a locally finite signed measure on $\R^{dk}$ defined for any bounded measurable sets $A_1,\ldots, A_k$ in $\R^d$ by
\begin{align*}
\gamma_{k}\left( \prod_{i=1}^k A_i \right) = \cum \left( \sum_{x\in \X} \1_{\lbrace x\in A_1 \rbrace}, \ldots , \sum_{x\in \X} \1_{\lbrace x\in A_k \rbrace} \right).
\end{align*}
\end{definition}

%

\begin{definition}\label{definition factorial cumulant measure}
 For $k\in \N$, the $k$-th order factorial cumulant moment measure $\gamma_{[k]}$ of a point process with factorial moment measure $\alpha^{(r)}$, for $r\leq k$, is a locally finite signed measure on $\R^{dk}$ defined for any bounded measurable sets $A_1,\ldots, A_k$ in $\R^d$ by
\begin{align*}
\gamma_{[k]}\left( \prod_{i=1}^k A_i \right) = \sum_{j=1}^k (-1)^{j-1} (j-1)! \ \sum_{B_1, \ldots, B_j \in \mathcal{P}_j^k } \prod_{i=1}^j \alpha^{\left(\left|K_i\right|\right)} \left( \prod_{k_i \in B_i} A_{k_i} \right), 
\end{align*}
 where for all $j\leq k$, $\mathcal{P}_j^k$ denote the set of all partitions of $\lbrace 1,\ldots,k \rbrace$ into $j$ non empty sets $B_1,\ldots,B_j$.
\end{definition}
For stationary point processes, we can define the so-called reduced version of the previous measure.

\begin{definition}\label{definition reduced measure}
 For any $k\geq 2$, the reduced $k$-th order  factorial cumulant moment measure $\gamma^{red}_{[k]}$ of a stationary point process is a locally finite signed measure on $ \R^{d(k-1)}$ defined for any bounded measurable sets $A_1,\ldots, A_k$ in $\R^d$ by 
 \begin{align*}
 \gamma_{[k]} \left( \prod_{i=1}^k A_i \right) = \int_{A_k} \gamma^{red}_{[k]} \left( \prod_{i=1}^{k-1} (A_i-x)  \right) dx
 \end{align*}
 where for $i=1,\ldots,k-1$, $A_i-x$ is the translation of the set $A_i$ by $x$.
\end{definition}
The reduced cumulant moment measure is defined similarly. An important property of signed measures is given by the following theorem leading to the definition of the total variation of a signed measure. 

\begin{theorem}[Hahn-Jordan decomposition, {see~\cite[Theorem 5.6.1]{dudley2002real}}]\label{theorem hahn jordan decomposition}
 For any signed measure $\nu$, there exist two measures $\nu^+$ and $\nu^-$ uniquely determined by $\nu$ such that at least one of them is finite and 
 \begin{align*}
 \nu=\nu^+ - \nu^-.
 \end{align*}
\end{theorem}

\begin{definition}\label{definition total variation}
Let $\nu$ be a signed measure with Hahn-Jordan decomposition $\nu=\nu^+-\nu^-$.
The total variation measure $|\nu|$ of $\nu$ is defined by
 \begin{align*}
  |\nu| = \nu^+ + \nu^-.
   \end{align*}
\end{definition}

Following Theorem~\ref{theorem hahn jordan decomposition},  for $k\geq 2$, we denote the Hahn-Jordan decomposition of the reduced $k$-th order moment factorial cumulant measure $\gamma^{red}_{[k]}=\gamma^{+red}_{[k]} - \gamma^{-red}_{[k]}$.  

\begin{definition}\label{definition brillinger}
 A point process is Brillinger mixing if, for $k\geq 2$, we have
 \begin{align*}
  \left|\gamma^{red}_{[k]}\right| \left( \R^{d(k-1)} \right)  <+\infty.
 \end{align*}
\end{definition}
The different moment measures of a point process $\X$ are related to the power series expansion of the Laplace and the probability generating functionals of $\X$. 
\begin{definition}\label{definition laplace transform}
 The Laplace functional $L_\X$ of a point process $\X$ is defined for any bounded measurable function $f$ that vanishes outside a bounded set of $\R^d$ by
 \begin{align*}
  L_\X(f) = \E\left(e^{- \sum_{x\in\X} f(x)}\right).
 \end{align*}
\end{definition}


\begin{definition}\label{definition pgf}
The probability generating functional of a point process $\X$ is defined for any function $h$ from $\R^d$ into $[0,1]$, such that $1-h$ vanishes outside a bounded set, by
\begin{align*}
 G_\X(h)=\E\left( \exp\left({\sum_{x\in\X} \log h(x)}\right) \right).
\end{align*}
\end{definition}
Notice that for any function $h$ defined as in Definition~\ref{definition pgf} and taking values within a closed subset of $(0,1]$, we have 
 \begin{align*}
  G_\X(h) = L_\X(-\log(h) ).
 \end{align*}

\begin{proposition}[{\cite[Section 9.5]{daleyvol2}}]\label{power expansion functional}
Let $\X$ be a point process with cumulant  moment measures $\gamma_{k}$ and factorial cumulant moment measures  $\gamma_{[k]}$.
Let $f$ and $\eta$ be bounded measurable functions on $\R^d$ that vanish outside a bounded set. Assume further that $\eta$ takes values in $[0,1]$. Then, for all $N\in\N$, we have the following power series expansions when $s\geq 0$ and $s\to 0$
 \begin{align*}
  \log L_\X(s f) &= \sum_{j=1}^{N} \frac{(-s)^j}{j!} \int f(x_1)f(x_2) \ldots f(x_j) \gamma_j(dx_1 \times dx_2 \times \ldots \times dx_j) +o(s^N),\\
  \log G_\X(1-s \eta) &= \sum_{j=1}^{N} \frac{(-s)^j}{j!}  \int \eta(x_1)\eta(x_2) \ldots \eta(x_j) \gamma_{[j]}(dx_1 \times dx_2 \times \ldots \times dx_j) +o(s^N).
 \end{align*}
\end{proposition}

%
%

 We conclude this section by giving the relation between $\gamma_{k}$ and $\gamma_{[k]}$. To this end, we recall the definition of the Stirling numbers of the first and second kind and refer to~\cite[Section 5.2]{daleyvol1} for a detailed presentation. 
For $x\in\R$ and $k\in\N$,  we denote by $x^{[k]} = x(x-1)\ldots(x-k+1) \1_{\{0\leq k \leq x\}}$ the falling factorial of $x$.
Assuming $k\leq x$ and $1\leq j\leq k$, the Stirling numbers of the first kind $D_{j,k}$ and of the second kind $\Delta_{j,k}$ are defined by the relations
 \begin{align*}
   x^{[k]}&= \sum_{j=1}^k  (-1)^{k-j} D_{j,k} \, x^j \quad\text{and}\quad
     x^{k}= \sum_{j=1}^k  \Delta_{j,k}\,  x^{[j]}.
 \end{align*}


\begin{proposition}\label{stirling fact cum vs cum}
 Let $A$ be a bounded set of $\R^d$. For any integer $k$, we have the relations 
 \begin{align*}
  \gamma_{[k]}( A^k) &= \sum_{j=1}^k (-1)^{k-j} D_{j,k} \gamma_j(A^j),\\
  \gamma_{k} (A^k)&= \sum_{j=1}^k \Delta_{j,k} \gamma_{[j]}(A^j).
 \end{align*}
\end{proposition}

\begin{proof}
We denote by $L$ and $G$ the Laplace and probability generating functionals of a point process with, for $k\geq1$, factorial cumulant moment and cumulant moment measures $\gamma_{[k]}$ and $\gamma_{k}$, respectively. 
By Proposition~\ref{power expansion functional}, for all $N\in \N$, we have as $s\to 0$, $s\geq 0$
\begin{align}\label{dev G1}
 \log G(1-s\1_{\lbrace s\in A \rbrace}) = \sum_{k=1}^{N} \frac{(-s)^k}{k!} \gamma_{[k]}(A^k) + o(s^N).
\end{align}
As noticed after Definition~\ref{definition pgf},
\begin{align*}
  \log G(1-s\1_{\lbrace s\in A \rbrace}) &= \log L(-\log(1-s\1_{\lbrace s\in A \rbrace}))= \log L(-\log(1-s)\1_{\lbrace s\in A \rbrace}).\end{align*}  
  Since  $s \sim -\log(1-s)$ as $s\to 0$, we have by Proposition~\ref{power expansion functional},
  \begin{align*}
  \log G(1-s\1_{\lbrace s\in A \rbrace}) &= \sum_{j=1}^{N} \frac{\left[ \log(1-s)  \right]^j}{j!} \gamma_j(A^j)+ o(s^N).
\end{align*}
By~\cite[(24.1.3.I.B)]{stegun} we deduce that
\begin{align}\label{dev G2}
   \log G(1-s\1_{\lbrace s\in A \rbrace}) &= \sum_{j=1}^{N}  \frac{\gamma_j(A^j)}{j!} j!\sum_{k=j}^{N} (-1)^{k-j} D_{j,k} \frac{(-s)^k}{k!} + o(s^N)\notag \\
                   &= \sum_{k=1}^{N} \frac{(-s)^k}{k!}  \sum_{j=1}^{k} (-1)^{k-j} D_{j,k} \gamma_j(A^j) + o(s^N).
\end{align}
We conclude  by identifying the coefficients in~\eqref{dev G1} and~\eqref{dev G2}. The proof of the second formula is similar,  starting with the other powers expansion in Proposition~\ref{power expansion functional} and using~\cite[(24.1.4.I.B)]{stegun} instead of~\cite[(24.1.3.I.B)]{stegun}

\end{proof}

\section{Main result}\label{section brillinger mixing of DPP}

In this section, we prove in Theorem~\ref{theoreme DPP brillinger} below that a DPP with kernel verifying the condition $\K(\rho)$ is Brillinger mixing. We recall that this mixing property involves the factorial cumulant moments of the DPP. It is not easy to deduce these moments from the initial Definition~\ref{definition factorial cumulant measure}. However, the power series expansion of the log-Laplace functional in Proposition~\ref{power expansion functional}, which is known for a DPP, allows us to derive a closed form expression for the factorial cumulant measures as stated in the following lemma.

%


\begin{lemma}\label{lemme egalite cumulant K}
Consider a DPP with kernel $C$ verifying condition $\K(\rho)$ and, for $k\in \N$, denote its $k$-th factorial cumulant moment measure by $\gamma_{[k]}$. For every measurable bounded set $A$ in $\R^d$ and $k\geq 2$, we have
\begin{align*}
 \gamma_{[k]}(A^k)= (-1)^{k+1} (k-1)! \int_{A^{k}}   C(x_2-x_1) \ldots C(x_1-x_k)  dx_1\ldots dx_k.
 \end{align*}
\end{lemma}
\begin{proof}
By~\cite[Proposition 3.9]{ShiraiTakahashi1:03}, we deduce that for any bounded set $A\subset \R^d$ and $s$ small enough,
 \begin{multline*}
 \log \left( L_\X (s \1_{\lbrace s\in A \rbrace}) \right) = \sum_{p=1}^{\infty} \frac{(-s)^p}{p!} \sum_{n=1}^p (-1)^{n+1} \sum_{\stackrel{p_1+\ldots+p_n=p}{p_1,\ldots,p_n \geq 1} }\frac{p!}{n \cdot p_1! p_2! \cdots p_n!}\\ \int_{A^n} C(x_2-x_1) \ldots C(x_1-x_n)dx_1\ldots dx_n.
\end{multline*}
Then, by Proposition~\ref{power expansion functional}, we have by the last equation that for all $p\in\N$ and any bounded set $A\subset \R^d$,
\begin{multline*}
\gamma_p(A^p) = \sum_{n=1}^p (-1)^{n+1} \sum_{\stackrel{p_1+\ldots+p_n=p}{p_1,\ldots,p_n \geq 1}  }\frac{p!}{n \cdot p_1! p_2! \cdots p_n!}\\ \int_{A^n}   C(x_2-x_1) \ldots C(x_1-x_n)  dx_1\ldots dx_n.
\end{multline*}
Thus, by Proposition~\ref{stirling fact cum vs cum}, we have for $k\geq 2$,
\begin{multline}\label{expression cumulant factorial}
 \gamma_{[k]}(A^k)= \sum_{p=1}^k (-1)^{k-p} D_{p,k} \sum_{n=1}^p (-1)^{n+1} \sum_{\stackrel{p_1+\ldots+p_n=p}{p_1,\ldots,p_n \geq 1}  }\frac{p!}{n \cdot p_1! p_2! \cdots p_n!} \\ 
       \hspace{1cm}  \int_{A^n}   C(x_2-x_1) \ldots C(x_1-x_n)  dx_1\ldots dx_n.
\end{multline}
By \cite[(24.1.2.I.B)]{stegun}, it is easily seen that 
\begin{align} \label{simplification multinomiale 1}
    \sum_{\stackrel{p_1+\ldots+p_n=p}{p_1,\ldots,p_n \geq 1}  }\frac{p!}{  p_1! p_2! \cdots p_n!}   &=    \sum_{p_1+\ldots+p_n=p }\frac{p!}{  p_1! p_2! \cdots p_n!} - \sum_{p_1+\ldots+p_{n-1}=p }\frac{p!}{ p_1! p_2! \cdots p_{n-1}!}\notag \\ 
          &= n^p - (n-1)^p.
\end{align}
By definition
\begin{align}\label{simplification multinomiale 2}
 \sum_{p=1}^k (-1)^{k-p} D_{p,k} (n^p - (n-1)^p) = n^{[k]} - (n-1)^{[k]}
\end{align}
which is null for every $n<k$. Therefore, by~\eqref{simplification multinomiale 1} and~\eqref{simplification multinomiale 2}, only the terms $n=k$ is non null in the sum~\eqref{expression cumulant factorial}.
\end{proof}

We are now in position to prove our main result.


\begin{theorem}\label{theoreme DPP brillinger}
A DPP with kernel verifying the condition $\K(\rho)$, for a given $\rho>0$, is Brillinger mixing.
\end{theorem}

\begin{proof}
For any  $t>0$, we have by taking $f=\1_{[-t,t]^d}$ in Definition~\ref{definition reduced measure},
\begin{align*}
   \gamma_{[k]}([-t,t]^{dk})   &= \int_{\R^d} \1_{\lbrace x\in [-t,t]^d \rbrace}  \gamma^{red}_{[k]} \left( \left( [-t,t]^{d}-x \right)^{k-1} \right) dx.
   \end{align*}
By Lemma~\ref{lemme egalite cumulant K} 
\begin{align}\label{egalite cumulant integrale C}
\int_{\R^d}  \1_{\lbrace x\in [-t,t]^d \rbrace}  \gamma^{red}_{[k]} \left( \left( [-t,t]^{d}-x \right)^{k-1} \right) dx = (-1)^{k+1} (k-1)!\  I_k(t) 
\end{align}
where for all $k\geq 1$ and $t>0$, $I_k(t):= \int_{[-t,t]^{dk}}   C(x_2-x_1)\ldots C(x_1-x_k)  dx_1\ldots dx_k$.
Since $C$ verifies the condition $\K(\rho)$, by  Mercer's theorem, see also \cite[Section 2.3]{lavancier_extended}, we have for all $t>0$,
\begin{align*}
 C(x-y) = \sum_{j\in \N} \lambda_j(t) \phi_j(x) \phi_j(y),\quad \forall (x,y) \in [-t,t]^d,
\end{align*}
where for all $j\in \N$, $\lbrace \phi_j\rbrace_{j\in \N}$ is an orthonormal basis of $L^2\left( [-t,t]^d \right)$ and  $\lambda_j(t)$ belongs to $[0,1]$ by  \cite[Theorem 4.5.5]{macchi1975coincidence}. 
 Then, by orthogonality of the basis $\lbrace \phi_j\rbrace_{j\in \N}$, we have for all $t>0$ and $k\geq 1$,
\begin{align}\label{egalite integral operateur}
  I_k(t) =\sum_{ j\in \N} \lambda^k_j(t) \leq  \sum_{ j\in \N} \lambda_j(t) = I_1(t)
\end{align}
where $I_1(t) =   \int_{[-t,t]^{d}}   \rho dx = O(t^d)$. Thus, by Theorem~\ref{theorem hahn jordan decomposition}, \eqref{egalite cumulant integrale C} and \eqref{egalite integral operateur}, there exists a constant $\kappa>0$ and $T>0$ such that for all $t\geq T$,
\begin{multline}\label{egalite 2 expression cumulant short}
 \left| \int_{\R^d}  \1_{\lbrace x\in [-t,t]^d \rbrace}\left[  \gamma^{+red}_{[k]} \left( \left( [-t,t]^{d}-x \right)^{k-1} \right)-  \gamma^{-red}_{[k]} \left( \left( [-t,t]^{d}-x \right)^{k-1} \right) \right] dx \right|\leq \kappa t^d.
\end{multline}
Henceforth, we assume $t\geq T$. By Theorem~\ref{theorem hahn jordan decomposition}, at least one of the measure $\gamma^{+red}_{[k]}$ or $\gamma^{-red}_{[k]}$ is finite. Let us assume without loss of generality that $\gamma^{-red}_{[k]}$ is finite. Thus, by \eqref{egalite 2 expression cumulant short} and the monotonicity of the measure $\gamma^{-red}_{[k]}$, we have
\begin{align}\label{majoration c+}
 \int_{[-t,t]^d} \gamma^{+red}_{[k]} \left( \left( [-t,t]^{d}-x \right)^{k-1} \right) dx\leq t^d \left( \kappa + 2^d\gamma^{-red}_{[k]} ((\R^d)^{k-1} ) \right),
\end{align}
so by positivity of $\gamma^{+red}_{[k]}$,
\begin{align}\label{utilisation positivite mesure}
\int_{ \left[ \frac{-t}{2}, \frac{t}{2} \right]^d} \gamma^{+red}_{[k]} \left( \left( [-t,t]^{d}-x \right)^{k-1} \right) dx  \leq   \int_{[-t,t]^d}  \gamma^{+red}_{[k]}\left( \left( [-t,t]^{d}-x \right)^{k-1} \right)dx.
\end{align}
Further, for all $ (x,y)\in \left[ \frac{-t}{2}, \frac{t}{2} \right]^{2d}$, $y+x \in [-t,t]^{d}$, so for all $x \in \left[ \frac{-t}{2}, \frac{t}{2} \right]^d$ we have $\left[ \frac{-t}{2}, \frac{t}{2} \right]^d \subset [-t,t]^d -x$.
It follows by~\eqref{utilisation positivite mesure} and the monotonicity of $\gamma^{+red}_{[k]}$ that
\begin{align}\label{utilisation monotonicity mesure}
 \int_{ \left[ \frac{-t}{2}, \frac{t}{2} \right]^d} \gamma^{+red}_{[k]}\left( \left[ \frac{-t}{2}, \frac{t}{2} \right]^{d(k-1)}\right)dx  \leq \int_{ \left[ \frac{-t}{2}, \frac{t}{2} \right]^d} \gamma^{+red}_{[k]} \left( \left( [-t,t]^{d}-x \right)^{k-1} \right) dx .
\end{align}
Hence by \eqref{majoration c+}-\eqref{utilisation monotonicity mesure}, we have
\begin{align*}
 \gamma^{+red}_{[k]} \left(\left[\frac{-t}{2},\frac{t}{2}\right]^{d(k-1)}\right)  \leq  \left( \kappa + 2^d\gamma^{-red}_{[k]} (\R^{d(k-1)} ) \right).
\end{align*}
By letting $t$ tend to infinity in the last equation, we see that $\gamma^{+red}_{[k]}$ is finite and so is $\left|\gamma^{red}_{[k]}\right|$ by Definition~\ref{definition total variation}, which concludes the proof. 
\end{proof}

\section{Statistical applications}\label{section statistical application brillinger}

Many applications of the Brillinger mixing property for point processes may be found in~\cite{heinrichstellaintegrated}, \cite{StellaKleinempirical2011}, \cite{HeinrichProkesova:10}, \cite{jolivetuppermoment} and \cite{jolivet88momentestimation}. We present in this section some of these applications for DPPs. We prove in Section~\ref{functional} a general central limit theorem for certain functionals of a DPP  that are involved in the asymptotic properties of standard estimators. As an example, we apply this result to the  estimator of the intensity of a DPP. Another important application concerns the asymptotic behavior of minimum contrast estimators for parametric DPPs, which will be the subject of a separate paper.  In Section~\ref{pcf}, we obtain the asymptotic  properties of the kernel estimator of the pcf of a DPP. In particular, we prove a central limit theorem for the pointwise estimator of the pcf and for its integrated squared error.

\subsection{Asymptotic behaviour of functionals of order \emph{p}}\label{functional}

We present an important consequence of the Brillinger mixing property, namely a central limit theorem 
for a wide class of functionals of the point process 
and the convergence of their moments. 
%
A first theorem was mentioned in~\cite{krickeberg1980} and proved in~\cite{jolivetTCL}. We present here a more general version that yields in particular the asymptotic normality of standard statistics as the natural estimator of the intensity of the process. These results apply to stationary DPPs under condition $\K(\rho)$ as explained and exemplified at the end of this section. 

\medskip

For a given set $D$ of $\R^d$, we denote by $\partial D$  the boundary of $D$.
\begin{definition}\label{definition regular set}
 A sequence of subsets $\lbrace D_n\rbrace_{n\in \N}$ of $\R^d$ is called regular if for all $n\in\N$, $D_n\subset D_{n+1}$, $D_n$ is compact, convex and there exist constants $\alpha_1$ and $\alpha_2$ such that
 \begin{align*}
  \alpha_1 n^d &\leq \Dn \leq \alpha_2 n^d, \\
  \alpha_1 n^{d-1} &\leq \mathcal{H}_{d-1} \left( \partial D_n \right) \leq \alpha_2 n^{d-1}
 \end{align*}
 where $\mathcal{H}_{d-1}$ is the $(d-1)$-dimensional Hausdorff measure.
\end{definition}
Note that any sequence of subsets as above grows to $\R^d$ in all directions. 
For $p \geq 1$, let $f_D$ be a function from $\R^{dp}$ into $\R$ that depends on a given set $D \subset \R^d$ and define for a stationary point process $\X$,
 \begin{align*}
   N_p\left(f_D \right) := \sum_{ (x_1,\ldots,x_p) \in \X^p} f_{D}(x_1,\ldots,x_p).
 \end{align*}
By letting the set $D$ in the last equation be a sequence of regular subsets $\lbrace D_n \rbrace_{n\in\N}$, we have under some suitable conditions on the function $f_{D_n}$, the following central limit theorem on the sequence $\left\lbrace N_p\left(f_{D_n}\right) \right\rbrace_{n\in\N}$. The proof is postponed to Section~\ref{proof tcl jolivet}.

\begin{proposition}\label{TCL jolivet modifie}
Let $\lbrace D_n\rbrace_{n\in \N}$ and $\lbrace \widetilde{D}_n \rbrace_{n\in\N}$ be two sequences of regular sets in the sense of Definition~\ref{definition regular set} such that $\frac{|\widetilde{D}_n|} {|D_n|}\xrightarrow{n \rightarrow + \infty} \kappa$ for a given $\kappa >0$. 
Assume that there exists  a bounded and compactly supported function $F$ from $\R^{d(p-1)}$ into $\R^+$ such that for all $n\in \N$ and $(x_1,\ldots,x_p) \in \R^{dp}$,
 \begin{align}\label{majoration ft cumulant}
  |f_{D_n}(x_1, \ldots, x_p)| \leq \frac{ 1}{|\widetilde{D}_n|}  \1_{\lbrace x_1\in D_n \rbrace} F(x_2-x_1,\ldots,x_p-x_1) .
 \end{align}
Assume further that the point process $\X$ is ergodic, admits moment of any order and is Brillinger mixing in the sense of Definition~\ref{definition brillinger}. Then, for all $k\geq 2$, we have
 \begin{align}\label{comportement asymptotic cumulant}
\cum_k \left( \sqrt{\Dn}  N_p\left(f_{D_n} \right) \right) =  O \left( \Dn^{1- \frac{k}{2}} \right).
 \end{align}
 Moreover, if  there exists $\sigma>0$ such that 
 \begin{align}\label{condition limite variance tcl jolivet}
  \mathrm{Var}  \left(\sqrt{\Dn} N_p\left(f_{D_n}\right)\right) \xrightarrow[n\rightarrow +\infty]{} \sigma^2,
 \end{align}
we have the convergence 
  \begin{align}\label{convergence tcl jolivet}
\sqrt{\Dn}\left[  N_p\left(f_{D_n}\right) - \E\left( N_p \left(f_{D_n} \right) \right) \right] \convl \mathcal{N}(0,\sigma^2)
 \end{align}
 and the convergence of all moments to the corresponding moments of $\mathcal{N}(0,\sigma^2)$.
\end{proposition}

By~\eqref{comportement asymptotic cumulant}, the variance given in~\eqref{condition limite variance tcl jolivet} is uniformly bounded with respect to $n\in\N$.
If $D_n$ and $f_{D_n}$ in~\eqref{condition limite variance tcl jolivet} are sufficiently generic, the convergence~\eqref{condition limite variance tcl jolivet} of the variance holds true. However, in the general case, it must be assumed. To check~\eqref{condition limite variance tcl jolivet} in applications, it is convenient to express the  variance in~\eqref{condition limite variance tcl jolivet} in terms of the factorial cumulant moment measures of $\X$.  In appendix, we detail this expression for the important situations $p=1$ and $p=2$ with $f_{D_n}(x_1,x_2) = 0$ for $x_1\neq x_2$, see Lemmas~\ref{calcul variance intensite}, \ref{calcul variance general} and \ref{calcul covariance 3 termes general}. 

\medskip

%
%
%

Proposition~\ref{TCL jolivet modifie} applies to stationary DPPs with kernel verifying $\K(\rho)$ provided \eqref{majoration ft cumulant} is verified. Indeed, Soshnikov in~\cite{Soshnikov:00} proved that a stationary DPP is ergodic. Moreover, a DPP admits moments of any order by definition and is Brillinger mixing under condition $\K(\rho)$ by Theorem~\ref{theoreme DPP brillinger}. 
 As a direct application  when $p=1$, we retrieve a result of \cite{SoshnikovGaussianLimit} giving the asymptotic normality of the estimator of the intensity of a DPP. 

\begin{corollary}\label{corollary convergence normal intensite DPP}
 Let $\X$ be a DPP with kernel verifying $\K(\rho)$ for a given $\rho>0$  and  $\lbrace D_n\rbrace_{n\in \N}$ be a family of regular sets. Define for all $n\in\N$,
\begin{align}\label{rhoest}
 \widehat{\rho}_n =\frac{1}{|D_n|} \sum_{x\in \X} \1_{\lbrace x \in D_n\rbrace}.
\end{align}
We have the convergence
\begin{align*}
\sqrt{\Dn} \left(  \wrho - \rho \right) \convl N(0,\sigma^2)
\end{align*}
where $\sigma^2 =  \lim_{n\rightarrow +\infty} Var  \left(\sqrt{\Dn} \wrho \right) =\rho -\int_{\R^d} C(x)^2 dx$.
\end{corollary}

The proof of this corollary follows by taking $p=1$ and $f_{D_n}(x)= \frac{1}{\Dn}\1_{\lbrace x\in D_n\rbrace}$ in Proposition~\ref{TCL jolivet modifie}. In this case, the assumption~\eqref{condition limite variance tcl jolivet} holds by Lemma~\ref{calcul variance intensite} and a straightforward calculus.

\subsection{Applications to the empirical pair correlation function.}\label{pcf}
We consider in this section the estimation of the pcf of a stationary and isotropic DPP in $\R^d$. In this setting $g(x,y)=g_0(r)$ depends only  on the Euclidean distance $r=|x-y|$. 
Let $\left\lbrace D_n \right\rbrace_{n\in \N}$ be a sequence of regular subsets of $\R^d$ in the sense of Definition~\ref{definition regular set},  $\lbrace b_n \rbrace_{  n\in\N   }$  a  sequence of positive real numbers, and  $k$ a function from $\R$ into $\R^+$. For  $n\in \N$ and $z\in\R^d$, we denote for short $D_n^{z} := D_n-z$ the translation of $D_n$ by $z$. For $r>0$, we consider the kernel estimator of $g_0(r)$
\begin{align}\label{formula estimator pcf partie brillinger}
  \gn(r) = \frac{1}{\sigma_d r^{d-1} \wrho^2} \sum_{\substack{(x,y) \in \X^2  \\ x \neq y}} \1_{\left\lbrace x\in D_n,\ y \in D_n\right\rbrace}   \ \frac{1}{b_n |D_n\cap D_n^{ x-y}|} k\left(\frac{r-|x-y|}{b_n}\right)
\end{align}
where $\wrho$ is given by \eqref{rhoest} and $\sigma_d= \frac{2\pi^{d/2}}{\Gamma\left(d/2 \right)}$ denotes the surface-area of the $d$-dimensional unit sphere. Some comments and details about this estimator may be found, for instance, in \cite[Section 4.3.5]{mollerstatisticalinference} or~\cite{heinrichasymptotic:13}.

The following proposition gives the asymptotic normality of the pointwise estimator $\gn(r)$ for $r>0$. Its proof, given in Section~\ref{section preuve clt gn}, is based on Proposition~\ref{TCL jolivet modifie}  and results from \cite{StellaKleinempirical2011} .

 \begin{proposition}\label{proposition convergence moment 1 et clt gn}
Let $\lbrace D_n \rbrace_{n\in\N}$ be a regular sequence of subsets of $\, \R^d$.  Assume that the sequence $\lbrace b_n\rbrace_{n\in\N}$ is such that $b_n^3 \Dn \rightarrow +\infty $ and $b_n^5 \Dn \rightarrow 0$. Let $k$ be a symmetric and bounded function with compact support included in $[-T,T]$, for a given $T>0$, and $\int_\R k(t) dt=  1$.
 Let $C$ be an isotropic twice differentiable kernel on $ \R^d\setminus \lbrace 0 \rbrace$ verifying $\K(\rho)$ for a given $\rho>0$.  
 Then, for all $r>0$, we have the convergence
  \begin{align*}
 \sqrt{b_n \Dn} \left( \gn(r) - g_0(r) \right) \convl N(0, \tau_r^2) 
 \end{align*}
 where $\tau_r^2 = 2 \rho^{-2} \frac{g_0(r)}{\sigma_d r^{d-1}} \sqrt{\int_{\R} k^2(t) dt} $.
 \end{proposition}

 
 
 \medskip
 
In addition to the previous result, we state the asymptotic normality of the integrated squared error of the estimator $\wrho^{\, 2} \gn$ where $\gn$ is defined in~\eqref{formula estimator pcf partie brillinger}. This quantity is the basis of an asymptotic goodness-of-fit test for stationary DPPs as presented in~\cite{heinrichstellaintegrated}. 
For all segment $I \subset \R^+ \setminus \lbrace 0\rbrace $ and $n\in \N$, denote $$\mathrm{ISE}_n(I) =\int_I \left( \wrho^{\,2} \gn(r) - \rho^2 g_0(r) \right)^2 dr.$$

\begin{proposition}\label{proposition erreur quadratique moyenne stella}
Let $\lbrace D_n \rbrace_{n\in\N}$ be a regular sequence of subsets of $\, \R^d$.     Assume that the sequence $\lbrace b_n\rbrace_{n\in\N}$ is such that  $b_n\rightarrow 0$ and $b_n \Dn \rightarrow +\infty $.  Let $k$ be a symmetric and bounded function with compact support included in $[-T,T]$, for a given $T>0$, and $\int_\R k(t) dt=  1$.
 Let $C$ be an isotropic twice differentiable kernel on $ \R^d\setminus \lbrace 0 \rbrace$ verifying $\K(\rho)$ for a given $\rho>0$.  Then, for all segment $I\subset \R^+ \setminus \lbrace 0 \rbrace$, we have as $n$ tends to infinity,
 \begin{align*}
b_n \Dn \E\left( \mathrm{ISE}_n(I) \right) = 2 \rho^2 \int_I \frac{g_0(r)}{ \sigma_d r^{d-1} } dr \int_\R k(t)^2 dt + O\left(b_n\right) + O(|D_n| b_n^5).
\end{align*}
If in addition $b_n^5 \Dn \rightarrow 0$ then 
\begin{align*}
 \sqrt{b_n} \Dn \left( \mathrm{ISE}_n(I) - \E\left( \mathrm{ISE}_n(I) \right) \right) \convl N(0,\tau^2)
 \end{align*}
where $\tau^2= \displaystyle 8 \rho^4 \int_I \left(  \frac{g_0(r)}{\sigma_d r^{d-1}} \right)^2  dr \int_{\R} ( k\ast k)^2(s) ds $ and $\ast$ denotes the convolution product.
\end{proposition}


Proposition~\ref{proposition erreur quadratique moyenne stella} 
is an application to the DPP's case of the results given in~\cite{heinrichstellaintegrated}. 
In addition to the Brillinger mixing, ensured by Theorem~\ref{theoreme DPP brillinger}, 
 and the properties of the sequence $\lbrace D_n \rbrace_{n\in\N}$, the authors need two additional assumptions. Namely, these assumptions are the locally uniform Lipschitz continuity of the first derivative of $g_0$ and a second assumption related to the densities of the reduced factorial cumulant measures. By \eqref{DPPpcf} and since $C$ is twice differentiable on $\R^d\setminus \lbrace 0 \rbrace $, the first derivative of $g_0$ is uniformly Lipschitz continuous on every compact sets in $\R^+ \setminus \lbrace 0\rbrace$ so the first assumption holds. The second assumption is verified by Lemma~\ref{lemma condition cumulant heinrich} below.  Consequently, Proposition~\ref{proposition erreur quadratique moyenne stella} 
 is proved by~\cite[Lemma 3.4]{heinrichstellaintegrated}  and~\cite[Theorem 3.5]{heinrichstellaintegrated}.

\begin{lemma}\label{lemma condition cumulant heinrich}
  Let be an isotropic DPP with kernel $C$ verifying the condition $\K(\rho)$, whose  reduced factorial cumulant moment measures of order  $3$ and $4$  have densities $c^{red}_{[3]}$ and $c^{red}_{[4]}$, respectively. For all compact set $K\subset \R^d$ and $\epsilon>0$, we have
  \begin{align}\label{hyp densite cumulant 3}
   \sup_{ \stackrel{(u,v)\in \R^{2d}}{\left(|u|,|v|\right) \in ({K^{\oplus \epsilon}})^2}}  \left|c^{red}_{[3]} (u,v) \right| <+ \infty
   \end{align}
   and
   \begin{align}\label{hyp densite cumulant 4}
\sup_{ \stackrel{(u,v)\in \R^{2d}}{\left( |u|,|v|\right) \in ({K^{\oplus \epsilon}})^2}}  \int_{\R^d} \left|c^{red}_{[4]} (u,w,v+w) \right| dw <+ \infty,
\end{align}
where $K^{\oplus \epsilon}=K + B(0,\epsilon)$ and $B(0,\epsilon)$ is the Euclidean ball centred at $0$ with radius $\epsilon$.
\end{lemma}

\begin{proof}
By \eqref{expression densite cumulant 3}-\eqref{expression densite cumulant 4} in Section~\ref{appendix brillinger}, we have for all $(u,v,w) \in \R^{3d}$,
\begin{align*}
 c^{red}_{[3]} (u,v) = 2 C(u)C(v)C(v-u)
\end{align*}
and
\begin{multline*}
   c^{red}_{[4]} (u,v,w) = -2 \left[  C(u)C(v)C(u-w)C(v-w) \right. \\ \left. + C(u) C(w) C(u-v) C(v-w) + C(v) C(w) C(u-v) C(u-w)   \right].
\end{multline*}
Notice that $K^{\oplus \epsilon}$ is compact and since $C$ verifies the condition $\K(\rho)$, it is continuous. Therefore,  by~\eqref{expression densite cumulant 3}, \eqref{hyp densite cumulant 3} holds immediately. Finally, \eqref{hyp densite cumulant 4} is verified by Cauchy-Schwarz inequality and \eqref{expression densite cumulant 4}. 
\end{proof}

\section{Proof of Proposition~\ref{TCL jolivet modifie}}\label{section proof TCL brillinger}

\subsection{Complement on the moments and cumulants of a point process}

We present here the necessary background to prove Proposition~\ref{TCL jolivet modifie}. 
Let $p$ and $k$ be two integers and $\X$ a point process that admits moments of any order. Consider, for $1\leq i \leq k$, the random variables
\begin{align}
N_p\left( \phi_i\right)= \sum_{(x_1,\ldots,x_p) \in \X^p} \phi_i(x_1,\ldots,x_p)
\end{align}
where for $i=1,\ldots,p$, $\phi_i$ is a function from $\R^{dp}$ to $\R$.

For $l,s \leq kp$ , denote $\mathcal{P}_l^{kp}$ (resp. $\mathcal{Q}_s^l$) the set of all partitions of $\lbrace 1,\ldots,kp \rbrace$ (resp. $\lbrace 1,\ldots,l\rbrace$) into $l$ (resp. $s$) non empty sets $p_1,\ldots,p_l$ (resp. $q_1,\ldots,q_s$).   For $r=1,\ldots,s$, denote $\beta_1,\ldots,\beta_{|q_r|}$ the elements of the set $q_r$ and  $|q|$ the cardinal of a given set $q$. 
Then, as proved by Jolivet in~\cite[p121-122]{jolivetTCL}, we have 
\begin{align}\label{moment Phi shorten form}
  E\left(N_p\left( \phi_1\right) \ldots  N_p\left( \phi_k\right) \right) = \sum_{l=1}^{kp} \sum_{\Pi_l  \in \mathcal{P}_l^{kp}} \sum_{s=1}^l \sum_{\chi_s^l\in \mathcal{Q}_s^l}  I_l( \Pi_l,\chi_s^l)
\end{align}
where for all $l,s\leq kp$,
\begin{multline}\label{moment Phi}
I_l( \Pi_l,\chi_s^l) = \int_{\R^{dl}}  \prod_{m=1}^{l}\prod_{j \in p_m}  \1_{\lbrace x_m =\theta_j \rbrace} \times \ldots \\ \times \prod_{i=1}^{k} \phi_i( \theta_{(i-1)p+1},\ldots,\theta_{ip}) \prod_{r=1}^s \gamma_{|q_r|} (dx_{\beta_1} \ldots dx_{\beta_{|q_r|}}).
\end{multline}
The introduction of the term $\theta $ is not easy to understand at first sight. For the sake of clarity, we give an example for $p=k=2$ and $\Pi_2:=\lbrace p_1,p_2 \rbrace$ a given partition of the set $\lbrace 1,2,3,4 \rbrace$ into 2 non empty sets, namely $p_1=\lbrace 1,4 \rbrace$ and $p_2=\lbrace 2,3 \rbrace$. In this case, we have
  \begin{align*}
   \prod_{i=1}^{k} \phi_i( \theta_{(i-1)p+1},\ldots,\theta_{ip}) = \phi_1( \theta_1, \theta _2) \phi_2( \theta_3,\theta_4).
  \end{align*}
Thus, by the last equation, we have
\begin{align*}
 \prod_{m=1}^{2}\prod_{j \in p_m}  \1_{\lbrace x_m =\theta_j \rbrace} \prod_{i=1}^{k} \phi_i( \theta_{(i-1)p+1},\ldots,\theta_{ip}) = \phi_1( x_1, x_2) \phi_2( x_2 ,x_1)
\end{align*}
and a similar calculus is done if, for $l=1,\ldots,4$, we choose another partition $\Pi_l$ of $\lbrace 1,2,3,4 \rbrace$. We can now describe completely  $\cum(N_p\left( \phi_1 \right) , \ldots,N_p\left( \phi_k\right)  )$.
\begin{theorem}[\cite{jolivetTCL}]\label{expression cumulant jolivet}
 The cumulant moment $\cum(N_p\left( \phi_1\right), \ldots, N_p\left( \phi_k\right))$ is equal to the sum of integrals $I_l( \Pi_l,\chi_s^l)$  in Formula~\eqref{moment Phi shorten form} that are indecomposable, i.e. that can not be decomposed as a product of at least two integrals. 
\end{theorem}

\subsection{Proof of Proposition~\ref{TCL jolivet modifie}}\label{proof tcl jolivet}

Assuming \eqref{comportement asymptotic cumulant} and \eqref{condition limite variance tcl jolivet}, the proposition is proved by~\cite[Theorem 1]{jansonTCLcumulant}. Let us check  \eqref{comportement asymptotic cumulant}.
By \cite[Chapter II, Section 12, Equation (37)]{shiryaev1995probability}, if $X$ and $Y$ are two independent random variables $\cum_k(X+Y)=\cum_k(X) + \cum_k(Y)$ and  the cumulant of order $k$ of a constant is null for $k\geq 2$. Consequently, for $k\geq2$,
\begin{align*}
 \cum_k\left(\sqrt{\Dn} \left[ N_p\left(f_{D_n}\right) - \E\left(N_p \left(f_{D_n} \right) \right) \right] \right) &=  \cum_k\left(\sqrt{\Dn} N_p\left(f_{D_n}\right)\right)\\ &= \Dn ^{\frac{k}{2}}  \cum_k\left(N_p \left(f_{D_n}\right)\right).
\end{align*}

By Theorem~\ref{expression cumulant jolivet}, for every $k\in \N$, $\cum_k\left(N_p\left(f_{D_n}\right) \right)$ is a finite sum of indecomposable integrals $I_l( \Pi_l,\chi_s^l)$. 
Thus, it is sufficient to prove that for any $k\geq 2$, each integral $\left|I_l( \Pi_l,\chi_s^l)\right|=  O\left( |D_n|^{1-k}\right)$. By~\eqref{moment Phi} we have
\begin{multline}
I_l( \Pi_l,\chi_s^l)=\int_{\R^{dl}}  \prod_{m=1}^{l}\prod_{j \in p_m}  \1_{\lbrace x_m =\theta_j \rbrace} \prod_{i=1}^{k} f_{D_n}(\theta_{(i-1)p+1}, \theta_{(i-1)p+2}, \ldots, \theta_{ip} ) \\ \times \prod_{r=1}^s \gamma_{q_r} (dx_{\beta_1} \ldots dx_{\beta_{|q_r|}}).
\end{multline}
Then, by Definition~\ref{definition total variation}, we obtain from the last equation that
\begin{multline*}
|I_l( \Pi_l,\chi_s^l)|\leq \int_{\R^{dl}}  \prod_{m=1}^{l}\prod_{j \in p_m}  \1_{\lbrace x_m =\theta_j \rbrace}\prod_{i=1}^{k}\left|f_{D_n}(\theta_{(i-1)p+1}, \theta_{(i-1)p+2}, \ldots, \theta_{ip} )\right| \\ \times \prod_{r=1}^s \left|\gamma_{q_r}\right| (dx_{\beta_1} \ldots dx_{\beta_{|q_r|}}).
\end{multline*}
Using~\eqref{majoration ft cumulant}, we get
\begin{multline}\label{majoration I 1 cumulant}
|I_l( \Pi_l,\chi_s^l)|\leq \frac{1}{|\widetilde{D}_n|^k} \int_{\R^{dl}}  \prod_{m=1}^{l}\prod_{j \in p_m}  \1_{\lbrace x_m =\theta_j \rbrace} \prod_{i=1}^{k} \1_{\lbrace \theta_{(i-1)p+1}\in D_n \rbrace}   \\ \times  F \left( \theta_{(i-1)p+2}-\theta_{(i-1)p+1},\ldots,\theta_{ip}-\theta_{(i-1)p+1}\right) \prod_{r=1}^s \left|\gamma_{q_r}\right| (dx_{\beta_1} \ldots dx_{\beta_{|q_r|}}).
\end{multline}

Let $||F||_{\infty}$ denotes the supremum of $F$ on $\R^{d(p-1)}$.
Since the function $F$ is bounded and compactly supported, there exist compacts $K_1,\ldots,K_{p-1}$  such that
\begin{align*}
\forall (x_1,\ldots,x_{p-1}) \in (\R^d)^{p-1}, \quad F(x_1,\ldots,x_{p-1}) \leq ||F||_{\infty} \1_{\lbrace x_1 \in K_1 \rbrace}\ldots \1_{\lbrace x_{p-1} \in K_{p-1} \rbrace}.
\end{align*}
Then, we deduce from~\eqref{majoration I 1 cumulant} that
\begin{multline}\label{majoration pour faire comme jolivet cumulant}
|I_l( \Pi_l,\chi_s^l)| \leq \left( \frac{ ||F||_\infty}{|\widetilde{D}_n|}\right)^k \int_{\R^{dl}}   \prod_{m=1}^l \prod_{j \in p_m}  \1_{\lbrace x_m =\theta_j \rbrace} \prod_{i=1}^k  \1_{\lbrace \theta_{(i-1)p+1}\in D_n \rbrace} \\ \prod_{\eta=1}^{p-1}  \1_{\lbrace (\theta_{(i-1)p+\eta+1} - \theta_{(i-1)p+1}) \in K_\eta \rbrace}  \prod_{r=1}^s |\gamma_{q_r}| (dx_{\beta_1} \ldots dx_{\beta_{|q_r|}}).
\end{multline}
Moreover, as already proved in~\cite[Section 4, Theorem 3]{jolivetTCL}, we have as $n$ tends to infinity,
\begin{multline}\label{result jolivet asymptotic cumulant}
 \int_{\R^{dl}}   \prod_{m=1}^l \prod_{j \in p_m}  \1_{\lbrace x_m =\theta_j \rbrace} \prod_{i=1}^k \1_{\lbrace \theta_{(i-1)p+1}\in D_n \rbrace}  \ldots \\ \ldots \prod_{\eta=1}^{p-1}  \1_{\lbrace (\theta_{(i-1)p+\eta+1} - \theta_{(i-1)p+1}) \in K_\eta \rbrace}  \prod_{r=1}^s |\gamma_{q_r}| (dx_{\beta_1} \ldots dx_{\beta_{|q_r|}}) = O \left( |D_n|\right).
\end{multline}
Since $\frac{|\widetilde{D}_n|}{|D_n|}\xrightarrow[n\rightarrow +\infty]{} \kappa$, the right hand term of~\eqref{majoration pour faire comme jolivet cumulant} is,  by~\eqref{result jolivet asymptotic cumulant}, asymptotically of order $\Dn^{1-k}$, which ends the proof.

\section{Proof of Proposition~\ref{proposition convergence moment 1 et clt gn}}\label{section preuve clt gn}

The proof is based on the following lemmas.
\begin{lemma}\label{lemma clt sur gn}
 Let $\lbrace D_n \rbrace_{n\in\N}$ be a regular sequence of subsets of $\, \R^d$.     Assume that the sequence $\lbrace b_n\rbrace_{n\in\N}$ is such that  $b_n\rightarrow 0$ and  $b_n^3 \Dn \rightarrow +\infty $. Let $k$ be a  symmetric and bounded function with compact support included in $[-T,T]$, for a given $T>0$, and $\int_\R k(t) dt =  1$.
 Let $C$ be an isotropic twice differentiable kernel on  $ \R^d\setminus \lbrace 0 \rbrace$ verifying  $\K(\rho)$ for a given $\rho>0$. 
 Then, for all $r>0$,  we have the convergence
 \begin{align*}
 \sqrt{b_n \Dn} \left( \wrho^{\, 2}\gn(r) - \E( \wrho^{\, 2} \gn(r) ) \right) \convl N(0, \kappa^2) 
 \end{align*}
 where $\kappa^2 = 2 \rho^2 \frac{g_0(r)}{\sigma_d r^{d-1}} \sqrt{\int_{\R} k^2(t) dt} $.
\end{lemma}

 \begin{lemma}\label{lemma convergence moment 1 pcf chapter brillinger}
 Under the same assumptions as in Proposition~\ref{proposition erreur quadratique moyenne stella}, for all segment $I\subset \R^+ \setminus \{0 \}$,  there exists a constant $M\geq 0$ such that 
 \begin{align*}
 \sup_{r \in I} \left| \E\left( \wrho^{\, 2} \gn(r) - \rho^2 g_0(r) \right) \right|  &\leq  b_n^2  M\rho^2  \int_\R t^2 |k(t)| dt.
 \end{align*}
 \end{lemma}

 The proofs of Lemmas~\ref{lemma clt sur gn}-\ref{lemma convergence moment 1 pcf chapter brillinger} are postponed to the end of this section.  Let us now prove Proposition~\ref{proposition convergence moment 1 et clt gn}. For all $n\in \N$ and $r>0$, we have 
 \begin{align}\label{decomposition pour clt g}
  \wrho^{\,2} \sqrt{b_n \Dn} ( \gn(r) - g_0(r) ) = A_n + B_n + C_n 
 \end{align}
where
\begin{align*}
 A_n &= \sqrt{b_n \Dn} \left[ \wrho^{\,2} \gn(r) - \E(\wrho^{\, 2} \gn(r) ) \right] \\
 B_n &= \sqrt{b_n \Dn} \left[  \E(\wrho^{\,2} \gn(r)) - \rho^2 g_0(r) \right]\\
 C_n &= \sqrt{b_n \Dn} g_0(r) \left[\rho^2 - \wrho^{\, 2}\right].
\end{align*}
By Lemma~\ref{lemma clt sur gn}  we have the convergence
\begin{align}\label{conv 1 pour clt g}
 A_n \convl N(0,\kappa^2)
\end{align}
and since $b_n^5 \Dn$ tends to $0$ as $n$ tends to infinity, we have by Lemma~\ref{lemma convergence moment 1 pcf chapter brillinger},
\begin{align}\label{conv 2 pour clt g}
 B_n \convP 0.
\end{align}
By Corollary~\ref{corollary convergence normal intensite DPP} and the delta method, we know that $\sqrt{|D_n|} (\hat \rho_n^2 - \rho^2)$ converges in distribution. Since $b_n\to 0$, we deduce that 
\begin{align}\label{conv 3 pour clt g}
 C_n \convP 0. 
\end{align}
Finally, by inserting \eqref{conv 1 pour clt g}-\eqref{conv 3 pour clt g} in~\eqref{decomposition pour clt g}, the proposition is proved by Slutsky's theorem and  the almost sure convergence of $\wrho^{\, 2}$ to $\rho^2$.

 \subsection{Proof of Lemma~\ref{lemma clt sur gn}}
 
 We need the following result.
\begin{lemma}\label{lemma fenetre minore}
 Let $r>0$ and $D$ a subset of $\, \R^d$ such that the Euclidean ball $B(0,r)$ is included in $D$. Then, for all $x \in B(0,r)$, we have $D^{\circleddash r} \subset  D \cap D^x $. 
\end{lemma}
\begin{proof}[Proof of Lemma~\ref{lemma fenetre minore}]
 By definition, $ D \cap D^x = \left\lbrace u \in D, \ u + x \in D    \right\rbrace$ and 
 \begin{align*}
   D^{\circleddash r} = \left\lbrace u \in D, \ \forall v \in B(0,r), \ u+v \in D   \right\rbrace.
 \end{align*}
Since $x\in B(0,r)$, we have the inclusion $D^{\circleddash r} \subset  D \cap D^x $.
\end{proof}

Define for all $n\in \N$ and $(x_1,x_2)\in \R^{2d}$, 
 \begin{align*}
  f_{D_n}(x_1, x_2) =  \1_{\left\lbrace x_1\in D_n,\ x_2 \in D_n\right\rbrace}   \ \frac{1}{|D_n\cap D_n^{ x_1-x_2}|} k\left(\frac{r-|x_1-x_2|}{b_n}\right).
 \end{align*}
 Notice by \eqref{formula estimator pcf partie brillinger} that
\begin{align}\label{relation g et f pour lemme clt sur gn}
b_n \sigma_d r^{d-1} \wrho^{\,2}  \gn(r) =  \sum_{\substack{(x_1,x_2) \in \X^2 \\ x_1\neq x_2}}   f_{D_n}(x_1, x_2).
\end{align}
The support of $k$ is  included in $[-T,T]$ so for any  $(x_1,x_2) \in \R^{2d}$,
\begin{align*}
 \left|k\left(\frac{r-|x_1-x_2|}{b_n}\right) \right| \1_{\left\lbrace   x_2 \in D_n\right\rbrace} &\leq \left|k\left(\frac{r-|x_1-x_2|}{b_n}\right) \right| \1_{\left\lbrace |x_1-x_2|<r + T b_n\right\rbrace} \\
                                                                                     &\leq \left| k\left(\frac{r-|x_1-x_2|}{b_n}\right)\right| \1_{\left\lbrace |x_1-x_2|<r+T\right\rbrace} 
\end{align*}
as soon as $b_n<1$ which we assume without loss of generality since $\lbrace b_n \rbrace_{n\in\N}$ tends to $0$. Then, by Lemma~\ref{lemma fenetre minore} and since $k$ is bounded, there exists $M>0$ such that for all $n\in \N$,
 \begin{align*}
 \left| f_{D_n}(x_1, x_2) \right| \leq \frac{ M \1_{\lbrace x_1\in D_n \rbrace} }{| D_n^{\circleddash(r+T)}|} \1_{ \lbrace |x_1-x_2| \leq r+T \rbrace } .
 \end{align*}  
Therefore, by \eqref{comportement asymptotic cumulant} in Proposition~\ref{TCL jolivet modifie} and~\eqref{relation g et f pour lemme clt sur gn}, we have for all $k\geq 3$,
 \begin{align*}
  \cum_k \left( \sqrt{\Dn} b_n \sigma_d r^{d-1} \wrho^{\,2}  \gn(r) \right) = O( |D_n|^{1-\frac{k}{2}} )
 \end{align*}
whereby  for all $k\geq3$,
\begin{align*}
\cum_k \left(  \sqrt{b_n \Dn}  \wrho^{\, 2}\gn(r) \right) = O\left( b_n^{-k/2} \Dn^{1-\frac{k}{2}}  \right) 
\end{align*}
which tends to $0$ when $n$ goes to infinity since $b_n^3 \Dn \rightarrow \infty$.
Further, the convergences of $\cum_k \big(  \sqrt{b_n \Dn}  \wrho^{\, 2}\gn(r) \big)$ for $k=1,2$ are proved in~\cite{StellaKleinempirical2011} under conditions that we have already verified after Proposition~\ref{proposition erreur quadratique moyenne stella}. Finally, Lemma~\ref{lemma clt sur gn} is proved by the cumulant method, see~\cite[Theorem 1]{jansonTCLcumulant}.

\subsection{Proof of Lemma~\ref{lemma convergence moment 1 pcf chapter brillinger}}
 By~\eqref{formula estimator pcf partie brillinger} and Defintion~\ref{definition factorial moment intro}, we have for all $n\in \N$ and $r\in I$,
\begin{align*}
 \E\left( \wrho^{\, 2} \gn(r) \right) &= \frac{\rho^2}{ \sigma_d r^{d-1} b_n} \int_{\R^d}\int_{\R^d} \frac{ \1_{\lbrace x\in D_n, y \in D_n \rbrace} }{\left| D_n \cap D_n^{x-y}\right|} k\left( \frac{r-|x-y|}{b_n}\right) \frac{\rho^{(2)}(x,y)}{\rho^2}dxdy\\
    &= \frac{\rho^2}{\sigma_d r^{d-1} b_n} \int_{\R^d}\int_{\R^d} \frac{ \1_{\lbrace x\in D_n, y \in D_n \rbrace}}{\left| D_n \cap D_n^{x-y}\right|} k\left( \frac{r-|x-y|}{b_n}\right) g_0(|x-y|)dxdy.
\end{align*}
To shorten, denote $k_{b_n}(.) = \frac{1}{b_n} k\big( \frac{.}{b_n} \big)$. By the substitution $z=x-y$ and since $y\in D_n^z$ if and only if $z \in D_n^y$, we obtain from the last equation that
\begin{align*}
    \E\left( \wrho^{\, 2} \gn(r) \right)  &=   \frac{\rho^2}{\sigma_d r^{d-1}}  \int_{\R^d}\int_{\R^d}  \frac{\1_{\lbrace y\in D_n^z \cap  D_n \rbrace}}{\left| D_n \cap D_n^{z}\right|} k_{b_n}( r-|z|) g_0( |z|)dzdy \\
		    &=  \frac{\rho^2}{\sigma_d r^{d-1}}   \int_{\R^d} k_{b_n}( r-|z|) g_0( |z|)dz.
\end{align*}
Converting this integral into polar coordinates and by symmetry of  $k$, we get 
\begin{align*}
    \E\left( \wrho^{\, 2} \gn(r) \right)  &=  \rho^2 \int_0^{+\infty} \left(\frac{t}{r}\right)^{d-1}k_{b_n}( r-t)g_0(t)dt\\
                                    &= \rho^2 \int_{-\frac{r}{b_n}}^{+\infty}  k\left( u \right) \left(\frac{r+u b_n}{r}\right)^{d-1} g_0(r+u b_n)du.
\end{align*}
For all $n$ large enough, we have for all $r\in I$ that $\frac{t}{b_n} \geq T$, hence 
\begin{align*}
 \E\left( \wrho^{\, 2} \gn(r) \right) &= \rho^2 \int_{-T}^{T}  k\left( u \right) \left(\frac{r+u b_n}{r}\right)^{d-1} g_0(r+u b_n)du.
\end{align*}
Assume that $I$ writes $[r_{min},r_{max}]$ for $r_{max}>r_{min}>0$ and define for $s\in \R^+$, $f(s) := \left(\frac{s}{r}\right)^{d-1}g_0(s)$.
Notice that $I^{\oplus Tb_n}:=[r_{min}-T b_n , r_{max}+ T b_n] \subset \R^+ \setminus \{0\}$ as soon as $n$ is large enough which we assume without loss of generality.  Since $g_0(.)$ is of class $\mathcal{C}^2$ on $I^{\oplus Tb_n}$, so is $f(.)$. Thus by Taylor-Lagrange expansion, we have 
 \begin{align}\label{egalite pcf taylor expansion parti brillinger}
  \E\left( \wrho^{\, 2} \gn(r) \right) = \rho^2 \int_{-T}^{T}  k\left( u \right) \left( f(r)+ f'(r) u b_n + \int_r^{r+u b_n} f''(s) (u b_n + r-s) ds  \right)du.
\end{align}
Since $k$ is symmetric, we have $\int_{-T}^{T} uk(u)du =0$. Moreover, 
\begin{align*}
 \sup_{s\in I^{\oplus Tb_n}} \left|f''(s)\right| \leq \frac{1}{r^{d-1}_{min}} \sup_{s\in I^{\oplus Tb_n}} \left| \left(s^{d-1}g_0(s) \right)'' \right|,
\end{align*}
showing that $f''(.)$ is uniformly bounded on $I^{\oplus Tb_n}$ by a constant $M$. 
Further, for all $n\in \N$ and $s\in [r,r+u b_n]$, $|u b_n +r-s| \leq |ub_n|$. Finally, since $\int_{-T}^T k(u)du=1$, by~\eqref{egalite pcf taylor expansion parti brillinger}, we obtain for $n$ large enough,
\begin{align*}
 \left| \E\left( \wrho^{\, 2} \gn(r) - \rho^2 g_0(r) \right) \right|
 &\leq  b_n^2  M\rho^2  \int_\R u^2 |k(u)| du, \quad \forall r\in I. 
\end{align*}

\section{Appendix}\label{appendix brillinger}

We gather here some results useful to compute the asymptotic variance in Proposition~\ref{TCL jolivet modifie} and Corollary~\ref{corollary convergence normal intensite DPP}.\medskip

Let $\X$ a stationary point process on $\R^d$ and $c_{[2]}^{red}$, $c_{[3]}^{red}$ and $c_{[4]}^{red}$  the densities of its factorial cumulant moment measures  of order $2$, $3$ and $4$, respectively, assuming they exist. If $\X$ is a DPP with kernel $C$ verifying the condition $\K(\rho)$, for a given $\rho>0$, then we deduce from Definitions~\ref{DPPdefinition intro} and~\ref{definition factorial cumulant measure} that for all $(u,v,w) \in \R^{3d}$,
\begin{align}
 c^{red}_{[2]}(u) =& - C^2(u), \label{expression densite cumulant 2}\\
 c^{red}_{[3]} (u,v) =& \ 2 \ C(u)C(v)C(v-u), \label{expression densite cumulant 3} \\
 c^{red}_{[4]} (u,v,w) =& -2 \big[  C(u)C(v)C(u-w)C(v-w) + C(u) C(w) C(u-v) C(v-w) \notag   
   \\&+ C(v) C(w) C(u-v) C(u-w)  \big]. \label{expression densite cumulant 4}
\end{align}

\begin{lemma}\label{calcul variance intensite}
 Let $f$ be a function from $\R^{d}$ into $\R$ that is bounded, measurable and compactly supported. Then we have
 \begin{align*}
  \mathrm{Var}  \left( \sum_{x \in \X} f(x) \right) = \int_{\R^{2d}} f(x)f(x+y) c^{red}_{[2]}(y) dxdy  + \rho \int_{\R^d} f^2(x)dx.
 \end{align*}
\end{lemma}

\begin{proof}
Notice that 
\begin{align*}
  \left( \sum_{x \in \X} f(x) \right) ^2 = \sum_{(x,y) \in \X^2} f(x)f(y) + \sum_{x \in \X} f^2(x).
\end{align*}
Then, denoting $\rho^{(2)}$ the density of the second order factorial moment measure, we have by Definitions~\ref{definition factorial moment intro} and~\ref{definition factorial cumulant measure},
 \begin{align*}
   \mathrm{Var}  \left( \sum_{x \in \X} f(x) \right) &= \int_{\R^{2d}} f(x)f(y) \left( \rho^{(2)}(x,y) - \rho^2\right) dxdy  + \rho \int_{\R^d} f^2(x)dx\\ 
                                                     &= \int_{\R^{2d}} f(x)f(y)c_{[2]}(x,y) dxdy  + \rho \int_{\R^d} f^2(x)dx.
 \end{align*}
Finally, by Definition~\ref{definition reduced measure}, we have
\begin{align*}
 \mathrm{Var}  \left( \sum_{x \in \X} f(x) \right) &= \int_{\R^{2d}} f(x)f(x+y) c^{red}_{[2]}(y) dxdy  + \rho \int_{\R^d} f^2(x)dx.
\end{align*}
\end{proof}

\begin{lemma}\label{calcul variance general}
Let $f$ be a function from $\R^{2d}$ into $\R$ that is bounded, measurable and compactly supported. Then, we have
 \begin{align*}
  &\mathrm{Var} \left( \sum_{(x,y) \in \X^2}^{\neq} f(x,y) \right) \\
  &= \int_{\R^{2d}} \left( f^2(x,x+y)  + f(x,x+y) f(x+y,x)  \right) c_{[2]}^{red}(y) dx dy \\
  &+ \rho^2 \int_{\R^{2d}} \left( f^2(x,y) + f(x,y)f(y,x) \right)dx dy \\
  &+ \int_{\R^{3d}}  \left[ f(x,x+y) + f(x+y,x)  \right] \left[ f(x+y,x+u)+ f(x+u,x+y) \right] c_{[3]}^{red}(y,u) dx dy du \\
  &+ 2\rho \int_{\R^{3d}} \left[f(x,y) + f(y,x) \right] \left[f(y,y+u) + f(y+u,y) \right] c_{[2]}^{red}(u)  dxdydu \\
  &+ \rho \int_{\R^{3d}} \left[f(x,y) + f(y,x) \right] \left[f(y,x+u) + f(x+u,y)  \right] c_{[2]}^{red}(u)  dxdydu \\
  &+ \rho^3  \int_{\R^{3d}} \left[f(x,y) + f(y,x) \right] \left[f(y,u) + f(u,y)  \right]  dxdydu \\
  &+ \int_{\R^{4d}}   f(x,x+y) f(x+u,x+v)  c_{[4]}^{red}(y,u,v) dx dy du dv  \\
  &+4 \rho \int_{\R^{4d}}   f(x,y) f(y+u,y+v)  c_{[3]}^{red}(u,v) dx dy du dv   \\
  &+ 2\int_{\R^{4d}}    f(x,y) f(x+u,y+v) c_{[2]}^{red}(u) c_{[2]}^{red}(v) dx dy du dv \\
  &+4 \rho^2 \int_{\R^{4d}}  f(x,y) f(x+u,v)  c_{[2]}^{red}(u)  dx dy du dv .
 \end{align*}
\end{lemma}

\begin{proof}
This lemma is a generalization of~\cite[Lemma 5]{HeinrichProkesova:10} for a function $f$ non necessary symmetric. The variance is first computed with respect to the factorial moment measure by Definition~\ref{definition factorial moment intro}. Then, the factorial moment measure is written in terms of the factorial cumulant moment measure by~\cite[Corollary 5.2 VII]{daleyvol1} and the result is  obtained by using Definition~\ref{definition reduced measure}. We  refer to the proof of~\cite[Lemma 5]{HeinrichProkesova:10} for the detailed calculus, the only change being the use of the following decomposition in place of the original one,
\begin{align*}
 \left( \sum_{(x,y)\in \X^2}^{\neq} f(x,y) \right) ^2  &= \sum_{(x,y)\in \X^2}^{\neq} f^2(x,y) + f(x,y)f(y,x)  \\
 &+\sum_{(x,y,u)\in \X^3}^{\neq}  \left[ f(x,y)+f(y,x) \right] \left[f(y,u) + f(u,y) \right]  \\
 &+ \sum_{(x,y,u,v)\in \X^4}^{\neq} f(x,y)f(u,v).
\end{align*}
\end{proof}

%
%
 
\begin{lemma}\label{calcul covariance 3 termes general}
Let $f$ be a function from $\R^{2d}$ into $\R$ that is bounded, measurable, symmetric and compactly supported. Let $h$ be a function from  $\R^{d}$ into $\R$ that is bounded, measurable and compactly supported. Then, we have
 \begin{align*}
  \mathrm{Cov} \left( \sum_{(x,y) \in \X^2}^{\neq} f(x,y) , \sum_{u\in \X} h(u) \right)  &= \int_{\R^{3d}}  f(x,x+y)h(u+x)   c_{[3]}^{red}(y,u) dx dy du \\
  &+ \rho \int_{\R^{3d}} f(x,y)h(u+x)    c_{[2]}^{red}(u) dx dy du \\
  &+ \rho \int_{\R^{3d}} f(x,y)h(u+y)    c_{[2]}^{red}(u) dx dy du \\
  &+  \int_{\R^{2d}} f(x,y+x)  \left[ h(x) +h(y+x) \right] c_{[2]}^{red}(y)  dxdy \\
  &+ \rho^2 \int_{\R^{2d}}f(x,y)  \left[ h(x) + h(y)  \right]  dxdy.
 \end{align*}
\end{lemma}
\begin{proof}
Notice that
\begin{align*}
 \left( \sum_{(x,y)\in\X^2}^{\neq} f(x,y) \right) \left( \sum_{u\in \X} h(u) \right) =  \sum_{(x,y,u)\in \X^3}^{\neq} f(x,y)h(u) +   \sum_{(x,y)\in\X^2}^{\neq} f(x,y) (h(x)+h(y)).
\end{align*}
Then, by the last equation and Definition~\ref{definition factorial moment intro}, we have
 \begin{align*}
  \mathrm{Cov} \left( \sum_{(x,y) \in \X^2}^{\neq} f(x,y) , \sum_{u\in \X} h(u) \right) &= \iiint f(x,y)h(u) \left(\rho^{(3)} (x,y,u) -\rho \rho^{(2)}(x,y)\right) dx dy du \\
  &+ \iint f(x,y) (h(x)+h(y) ) \rho^{(2)}(x,y) dxdy.
 \end{align*}
Finally,  the proof is concluded by~\cite[Corollary 5.2 VII]{daleyvol1} and Definition~\ref{definition reduced measure}.
\end{proof}

\bibliographystyle{acm}
 \bibliography{biblio_article_brillinger_biscio.bib}

\end{document}